\documentclass{amsart}

\usepackage{amsmath,amsfonts,amsthm,graphicx}
\usepackage[dvips]{color}

\newtheorem{theorem}{Theorem} [section]
\newtheorem{corollary} [theorem] {Corollary}
\newtheorem{prop} [theorem] {Proposition}
\newtheorem{lemma} [theorem] {Lemma}

\theoremstyle {definition}

\theoremstyle{remark}

\newtheorem{remark}[theorem]{Remark}
\newtheorem{example}[theorem]{Example}

\newcommand{\alert}{}

\renewcommand{\Re}{\operatorname{Re}}
\newcommand{\dd}{ { \mathrm{d}} }
\newcommand{\TV}{\mathrm{TV}}
\DeclareMathOperator{\supp}{supp}

\newcommand{\R}{\mathbb R}

\title[Higher order Gr\"unwald approximations]{Higher order Gr\"unwald approximations of fractional derivatives and fractional powers of operators}

\author[B. Baeumer]{Boris Baeumer}

\author[M.~Kov\'acs]{Mih\'aly Kov\'acs}

\author[H. Sankaranarayanan]{Harish Sankaranarayanan}


\begin{document}

\begin{abstract}
We give stability and consistency results for higher order Gr\"unwald-type formulae used in the approximation of solutions to fractional-in-space partial differential equations. We use a new Carlson-type inequality for periodic Fourier multipliers to gain regularity and stability results. We then generalise the theory to the case where the first derivative operator is replaced by the generator of a bounded group on an arbitrary Banach space.
\end{abstract}

\date{\today}

\maketitle

\section{Introduction}
\label{intro}
In a series of articles Meerschaert, Scheffler and Tadjeran \cite{Meerschaert2006f,Tadjeran2007,Meerschaert2004d,Meerschaert2006e,Tadjeran2006} explored consistency and stability for numerical schemes for fractional-in-space partial differential equations using a shifted Gr\"unwald formula to approximate the fractional derivative. In particular, in \cite{Tadjeran2006} they showed consistency if the order of the spatial derivative is less or equal to 2. They obtained a specific error term expansion for $f\in C^{4+n}(\R)$, where $n$ is the number of error terms, as well as stability for the Crank-Nicolson scheme using Gershgorin's theorem to determine the spectrum of the Gr\"unwald matrix.  Richardson extrapolation then gave second order convergence in space and time of the numerical scheme.

{\alert In this article we explore convergence with error estimates for higher order Gr\"unwald-type approximations of semigroups generated first by a fractional derivative operator on $L^1(\R)$ and then, using a transference principle, by fractional powers of group or semigroup generators on arbitrary Banach spaces.}

It was already shown in \cite[Proposition 4.9]{Baeumer2009a}  that for all
\begin{equation*}
f\in X_{\alpha}(\mathbb{R}):=\{f\in L_1(\R):\exists g\in L_1(\R) \;\mathrm{ with }\;\hat g(k)=(-ik)^\alpha \hat f(k), k\in\R\}
 \end{equation*}
 the first order Gr\"unwald scheme
 \begin{equation}\label{A_hIntro}
A^\alpha _{h,p} f(x)=\frac{1}{\Gamma (-\alpha)} \frac {1}{h^\alpha} \sum_{m=0}^\infty \frac{\Gamma(m-\alpha )}{\Gamma (m+1)} f(x-(m-p)h)
\end{equation}
converges in $L_1(\R)$ to $f^{(\alpha)}$ as $h\to 0+$ and any shift $p\in \R$.
Here
$\hat g(k)=\int_{-\infty}^\infty e^{ikx}g(x)\,dx$ denotes the Fourier transform of $g$ and for $f\in X_\alpha(\R)$, $f^{(\alpha)}=g$ iff $(-ik)^\alpha\hat f=\hat g$. In Section 3 we develop higher order Gr\"unwald-type approximations $\tilde A_h^\alpha $.  In Corollary \ref{HigherOrderConsistency} we can then give the consistency error estimate
\begin{equation}\label{introhighOrder}\left\|\tilde A_h^\alpha f-f^{(\alpha)}\right\|_{L_1(\R)}\le Ch^{n}\|f^{(\alpha+n)}\|_{L_{1}(\R)}\end{equation}
for an $n$-th order scheme.

{\alert Using a new Carlson-type inequality for periodic multipliers developed in Section 2 (Theorem \ref{thm:eta}) we investigate the stability and smoothing of certain approximation schemes $\tilde A_h^\alpha$ in Section 4.} The main tool is Theorem \ref{generalresult} which gives a sufficient condition for multipliers associated with difference schemes approximating the fractional derivative to lead to stable schemes with desirable smoothing. In particular, we show in {\alert Proposition \ref{thm:spg}} that stability for a numerical scheme using \eqref{A_hIntro} to solve the Cauchy problem
\begin{equation}\label{CPx}\frac{\partial }{\partial t}u(t,x)=(-1)^{q+1}\frac{\partial^\alpha}{\partial x^\alpha} u(t,x);u(0,x)=f(x)\end{equation} with $2q-1<\alpha<2q+1,q\in \mathbb N$ can only be achieved for a unique shift $p$; i.e.  it is necessary that $p=q$ for $(-1)^{q+1}A_{h,p}^\alpha$ to generate bounded semigroups on $L_1(\R)$ where the bound is uniform in $h$. Furthermore, in Theorem \ref{4.5} we prove stability and smoothing of a second order scheme.

Developing the theory in $L_1$ allows {\alert in Section 5} the transference of the theory to fractional powers of the generator $-A$ of a strongly continuous (semi-)group $G$ on a Banach space $(X,\|\cdot\|)$, noting that $f(x-(m-p)h)$ in \eqref{A_hIntro} will read as $G((m-p)h)f$ \cite{Baeumer2009a}. The abstract Gr\"unwald approximations with the optimal shifts generate
  analytic semigroups, uniformly in $h$, as shown in Theorem \ref{groupthm}.
  This is the main property needed in Corollary \ref{highOrder} to show that the error between the solution $S_\alpha(t)f=e^{t(-1)^{q+1}A^\alpha}f$ and a fully discrete solution $u_n$ obtained via a Runge-Kutta method with stage order $s$, order $r\ge s+1$ and an $N+1$ order Gr\"unwald approximation is bounded by
$$
\|S_{\alpha}(t)f-u_n\|\le C\left(n^{-r}\|f\|+h^{N+1}\left|\log\frac{t}{h^{\alpha}}\right|\,\|A^{N+1}f\|\right),~h>0,~t=n \tau.
$$
Note that this yields error estimates of our numerical approximation schemes applied to \eqref{CPx} in spaces where the translation semigroup is strongly continuous, such as $L_p(\R)$, $1\le p<\infty$, $BUC(\R)$, $C_0(\R)$, etc. Using the abstract setting we can also conclude that the consistency error estimate \eqref{introhighOrder} holds in those spaces, with the $L_1$ norm replaced by the appropriate norm.

{\alert Finally, in Section 6,  we give results of some numerical experiments, including a third order scheme, highlighting the efficiency of the higher order schemes and the dependence of the convergence order on the smoothness of the initial data.}

\section{Preliminaries}\label{sec:pre}
 A measurable function $\psi : \mathbb{R} \rightarrow \mathbb{C}$ is called an \textit{$L_1$-multiplier}, if for all $f \in L_1=L_1(\mathbb{R})$ there exists a $v \in L_1$ such that $ \psi \hat{f} = \hat{v}$ where $\hat{f}(k)=\int_{-\infty} ^\infty e^{ikx}f(x)\,\dd x$ denotes the Fourier transform of $f$. Define an operator, using the uniqueness of Fourier transforms, $T_\psi:L_1 \rightarrow L_1$ by $T_\psi f:= v$ where $v$ is defined as above. It is well known that $T_\psi$ is a closed operator and since it is everywhere defined, by the closed graph theorem, it is bounded. Moreover, if $\psi$ is an $L_1$-multiplier then $\psi(k) = \hat{\mu}(k)=\int_{-\infty} ^\infty e^{ikx}\,\dd \mu(x)$ for some bounded Borel measure $\mu$ and $\left\|T_\psi\right\|_{\mathcal{B} ({L}_1)} = \left\|\mu\right\|_{\TV}$ where $\TV$ refers to the total variation of the measure and $\left\|\cdot\right\|_{\mathcal{B} ({L}_1)}$ is the operator norm on $L_1$. Furthermore, if the measure $\mu$ has a density distribution $g$ then
\begin{equation}
\label{normdensity}
\left\|T_\psi\right\|_{\mathcal{B} (L_1)} = \left\|g\right\|_{L_1}=\left\|\check{\psi}\right\|_{L_1},
\end{equation}
where $\check{\psi}(x)=\frac{1}{2\pi}\int_{-\infty} ^\infty e^{-ikx}g(k)\,\dd k$ denotes the inverse Fourier transform of $\psi.$ In the sequel we will also make use of the following scaling property. If $a\in \mathbb{R}, ~h>0,$ and
$\psi^a(k)=\psi(k+a)$ and $\psi_h (k)= \psi(hk)$, then
\begin{equation}
\label{multipliernormequivalence}
\left\|T_\psi\right\|_{\mathcal{B} (L_1)}=\left\|T_{\psi_{h}}\right\|_{\mathcal {B} (L_1)}=\left\|T_{\psi^a}\right\|_{\mathcal {B} (L_1)}.
\end{equation}

The following inequality is of crucial importance in our error analysis. It is a special case of a more general Carlson type inequality, see \cite[Theorem 5.10, p.107]{Carlsonbook}. We give an elementary proof here to keep the presentation self contained. The case $r=2$ is referred to as Carlson-Beurling Inequality, and for a proof see \cite[p.429]{olivebook} and \cite{Car}.

The space $W_r^1(\mathbb{R})$, $r\ge 1$, denotes the Sobolev space of
$L_r(\mathbb{R})$-functions with generalised first derivative in $L_r(\mathbb{R})$; that is, $f\in W_r^1(\mathbb{R})$ if $f\in L_r(\mathbb{R})$, $f$ is locally absolutely continuous, and $f'\in L_r(\mathbb{R})$.

\begin{prop}[Carlson-type inequality]
\label{carlson'sinequality}
If $g\in W_r^1(\mathbb{R})$, $1<r\le 2$, then there exists $f \in L_1$ and a constant $C=C(r)>0$, independent of $f$ and $g$, such that $\hat{f}=g\; \text{a.e.}$ and
\begin{equation}\label{eq:car}
\left\| f \right\|_{L_1}\leq C \left\| g \right\|^{\frac{1}{s}}_{L_r} \left\| g'\right\|^{\frac{1}{r}}_{L_r},
\end{equation}
where $\frac{1}{r}+\frac{1}{s}=1.$
\end{prop}
\begin{proof}
First, since $g \in L_r(\mathbb{R}),$ $f := \check{g} \in L_s(\mathbb{R}).$ Also recall the Hausdorff-Young-Titchmarsh inequality \cite[p.211]{greenfourierbook}
\begin{equation}
\label{HYT}
\left\|\hat u \right\|_{L_s} \leq \sqrt{2 \pi} \left\| u \right\|_{L_r}, u \in L_r, ~1\le r\le 2,~\frac{1}{r}+\frac{1}{s}=1.
\end{equation}
Note that $\left\|f\right\|_{L_1}=\left\|\check{g}\right\|_{L_1}=\frac{1}{2 \pi} \left\|\hat{g}\right\|_{L_1}$ and $\hat {g'}(x)=(-ix) \hat g(x).$ If $g\equiv 0$ there is nothing to prove. Otherwise, using H\"older's inequality, \eqref{HYT}, and setting $v=\frac{\left\|g'\right\|_{L_r}}{(r-1)^{1/r}\left\|g\right\|_{L_r}}$ we get
\begin{align*}
\left\| f \right\|_{L_1}= \frac{1}{2 \pi} \left\| \hat{g} \right\|_{L_1}
											&	=  \frac{1}{2 \pi}\left(\int_{\left|x\right|\leq v} \left|\hat{g} (x)\right| dx +\int_{\left|x\right| > v} \left|\frac{1}{x} (x \hat{g} (x))\right| dx \right)\\
												&\leq \frac{2^{1/r}}{2 \pi}\left( v^{\frac{1}{r}}\left\|\hat{g}\right\|_{L_s}+ v^{\frac{-1}{s}}\left\|(\cdot)\hat{g}(\cdot)\right\|_{L_s}(r-1)^{-1/r} \right)\\
												&\leq  \frac{2^{1/r}}{\sqrt{2 \pi}}\left( v^{\frac{1}{r}}\left\|g\right\|_{L_r}+ v^{\frac{-1}{s}}\left\|g'\right\|_{L_r}(r-1)^{-1/r}\right)\\
											&	= \sqrt{\frac{2}{\pi}}\frac{2^{1/r}}{(r-1)^{1/r^2}}\left\|g\right\|_{L_r}^{\frac{1}{s}}\left\|g'\right\|_{L_r}^{\frac{1}{r}}.
\end{align*}
Thus $f \in L_1,$ $\hat{f}$ exists and $g=\hat{\check{g}}= \hat{f} \; \text{a.e.}$
\end{proof}
\begin{remark}
Inequality \eqref{eq:car} can be rewritten in multiplier notation as
$$
\|T_{g}\|_{\mathcal{B}(L_1)}\le  C \left\| g \right\|^{\frac{1}{s}}_{L_r} \left\| g'\right\|^{\frac{1}{r}}_{L_r}.
$$
\end{remark}

The reach of Carlson's inequality can be greatly improved by the use of a partition of unity:
\begin{corollary}\label{partUnity}

Let $\phi_j$ be such that $\sum_j \phi_j(x)=1$ for almost all $x$. If $ g\phi_j\in W_r^1(\R)$, $1<r\le 2$, for all $j$ and $\sum_j\|g\phi_j\|_{L_r}^{\frac{1}{s}} \|(g \phi_j)'\|^{\frac{1}{r}}_{L_r}<\infty$, where $\frac{1}{r}+\frac{1}{s}=1$, then there exists $f\in L_1(\R)$ and a constant $C=C(r)>0$ independent of $f$ and $g$, such that $\hat f=g\; \text{a.e.}$ and
$$\|f\|_{L_1}\le C \sum_j \|g\phi_j\|_{L_r}^{\frac{1}{s}}\|(g\phi_j)'\|_{L_r}^{\frac{1}{r}}.$$
\end{corollary}

\begin{proof}
  By design, $g=\sum_j  g\phi_j$. Let $f_j\in L_1(\mathbb{R})$ be such that $\hat f_j=g\phi_j$ by Proposition \ref{carlson'sinequality}. By assumption
\begin{equation*}
\left\|\sum_{j=1}^n f_j\right\|_{L_1} \le \sum_{j=1}^n\left\|f_j\right\|_{L_1}\le C \sum_j \|g\phi_j\|_{L_r}^{\frac{1}{s}} \|(g \phi_j)'\|^{\frac{1}{r}}_{L_r}<\infty.
\end{equation*}
So $\sum_{j=1}^\infty f_j$ converges to some $f\in L_1,$ $\|f\|_{L_1}\le  C(r)\sum_j \|g\phi_j\|_{L_r}^{\frac{1}{s}}\|(g\phi_j)'\|_{L_r}^{\frac{1}{r}}$ and $\hat f=\sum_j \hat f_j=g\; \text{a.e.},$
as $\lim_{n\to \infty} \left\|\hat{f}-\sum_{j=1}^n \hat{f_j}\right\|_{L_1} \le \lim_{n\to \infty} \left\|f- \sum_{j=1}^n f_j\right\|_{L_1}= 0.$
\end{proof}

\subsection{Periodic multipliers}
For periodic multipliers a Carlson type inequality is not directly applicable as these are not Fourier transforms of $L_1$-functions. We mention that in \cite{besov} a suitable smooth cut-off function $\eta$ was used where $\eta=1$ in a neighborhood of $[-\pi,\pi]$ and $\eta$ has compact support to estimate the multiplier norm of a periodic multiplier $\psi$ by the non-periodic one $\eta \psi$. For the multiplier norm of $\eta\psi$ the above Carlson's type inequality can be then used.
However, we prove a result similar to Proposition \ref{carlson'sinequality} for periodic multipliers which makes the introduction of a cut-off function superfluous and hence simplifies the technicalities in later estimates.

The space $W^1_{r,per}[-\pi,\pi]$ below denotes the Sobolev space of $2\pi$-periodic functions $g$ on $\mathbb{R}$ where both $g$ and its generalized derivative $g'$ belong to $L_r[-\pi,\pi]$.

\begin{theorem}
\label{thm:eta}
Let $g\in W^1_{r,per}[-\pi,\pi],$ $1<r\le 2,$ then $g$ is an $L_1$-multiplier and there is $C=C(r)>0$, independent of $g$, such that
$$\|T_{g}\|_{\mathcal{B}(L_1)}\leq \left|a_0  \right| + C \left\| g \right\|^{\frac{1}{s}}_{L_r[-\pi,\pi]} \left\| g'\right\|^{\frac{1}{r}}_{L_r[-\pi,\pi]},$$ where $\frac{1}{r}+\frac{1}{s}=1$ and $a_0=\frac{1}{2 \pi} \int_{-\pi} ^{\pi} g(x) dx$, denotes the $0^{\text{th}}$ Fourier coefficient of $g.$
\end{theorem}
\begin{proof}
Since $g\in L_r[-\pi,\pi]$ and is $2\pi$-periodic it can be written as a Fourier series
$g(x)=\sum_{k=-\infty}^{\infty}a_ke^{ikx}$, where $a_k=\frac{1}{2 \pi} \int_{-\pi} ^{\pi}e^{-ikx} g(x) dx$, denotes the $k^{\text{th}}$ Fourier coefficient of $g$. If $\mu:=\sum_{k=-\infty}^{\infty}a_k\delta_k$, where $\delta_k$ is the Dirac measure at $k$, then $g=\hat{\mu}$ and hence $g$ is an $L_1$-multiplier if and only if $\|T_{g}\|_{\mathcal{B}(L_1)}=\|\mu\|_{\TV}=\sum_{k=-\infty}^{\infty}|a_k|<\infty$. First, note that $\left|a_0  \right|\le \frac{1}{2 \pi} \int_{-\pi} ^{\pi} \left|g(x) \right|dx < \infty$ and that $ik a_k$ are the Fourier coefficients of $g'$. Using Bellman's inequality \cite{Bellman} with $\alpha= \beta= s$, and Hausdorff-Young inequality, (see \cite{Hausdorff,Young}), we have
$$\sum_{k=1} ^\infty \left|a_k\right| \le C \left( \sum_{k=1}^\infty \left|a_k\right|^s \right)^{\frac{1}{s^2}} \left( \sum_{k=1}^\infty \left|(ik a_k)\right|^s \right)^{\frac{1}{s}\left(\frac{s-1}{s}\right)}\le C \left\|g\right\|_{L_r[-\pi,\pi]}^{\frac{1}{s}} \left\|g'\right\|_{L_r[-\pi,\pi]}^{\frac{1}{r}}.$$ Clearly, the same inequality holds for $\sum_{k=-\infty} ^{-1} \left|a_k\right|$. Thus $$\sum_{k=-\infty}^{\infty}|a_k|\le \left|a_0  \right|+ C \left\|g\right\|_{L_r[-\pi,\pi]}^{\frac{1}{s}} \left\|g'\right\|_{L_r[-\pi,\pi]}^{\frac{1}{r}} < \infty$$ and the proof is complete.
\end{proof}

\begin{remark}
The term $|a_0|$ cannot be removed from the above estimate in general as the specific example $g \equiv 1$ shows.
\end{remark}

\section {Consistency: Higher order Gr\"unwald-type formulae}

Let $\alpha \in \mathbb{R}_+$ and let $$X_\alpha(\mathbb{R}):=\{f \in L_1=L_1(\mathbb{R}): \exists \;g \in L_1\; \text{with} \;\hat{g} (k)= (-ik)^\alpha \hat{f} (k),~k\in \mathbb{R}\}.$$ For $f \in X_\alpha(\mathbb{R})$ define $f^{(\alpha)}=g$ if $g \in L_1$ and $(-ik)^\alpha \hat{f} (k)= \hat{g} (k)$ for $k\in \mathbb{R}$. The function $f^{(\alpha)}$, defined uniquely by the uniqueness of the Fourier transform, is called the \textit{Riemann-Liouville fractional derivative} of $f$. Set $\left\|f\right\|_\alpha:=\left\|f^{(\alpha)}\right\|_{L_1(\mathbb{R})}$ for $f \in X_\alpha(\mathbb{R}).$ Similarly, we define $$X_\alpha(\mathbb{R}_+):=\{f \in L_1(\mathbb{R}_+): \exists g \in L_1(\mathbb{R}_+)\; \text{with} \;\hat{g} (z)= (-z)^\alpha \hat{f} (z),~\Re z\le 0\},$$ where $\hat{f}(z)=\int_0^{\infty}e^{zt}f(t)\,\dd t$ denotes the Laplace transform of $f$. For $f \in X_\alpha(\mathbb{R}_+)$ define $f^{(\alpha)}=g$ if $g \in L_1(\mathbb{R}_+)$ and $(-z)^\alpha \hat{f} (z)= \hat{g} (z)$ for $\Re z\le 0$.  Set $\left\|f\right\|_\alpha:=\left\|f^{(\alpha)}\right\|_{L_1(\mathbb{R}_+)}$ for $f \in X_\alpha(\mathbb{R}_+).$ \\
In order to calculate the fractional derivative of $f$, a shifted Gr\"{u}nwald formula was introduced in \cite{Meerschaert2004d} given by
\begin{equation}
\label{grunwald}
A^\alpha _{h,p} f(x)=\frac{1}{\Gamma (-\alpha)} \frac {1}{h^\alpha} \sum_{m=0}^\infty \frac{\Gamma(m-\alpha )}{\Gamma (m+1)} f(x-(m-p)h),
\end{equation}
where $p$ is an integer.
\begin{remark}\label{rem:l1r+}
When \eqref{grunwald} is applied to a function $f\in L_1(\R_+)$, we extend $f$ to $L_1(\R)$ by setting $f(x)=0$ for $x<0$. Hence, if $p\le 0$, then with this convention, $A_{h,p}^\alpha f$ is supported on $\R_+$ and hence $A_{h,p}^\alpha$ can be regarded as an operator on $L_1(\R_+)$.
\end{remark}
In \cite[Proposition 4.9]{Baeumer2009a} it is shown that for all $f\in X_{\alpha}(\mathbb{R})$ we have $A^\alpha _{h,p} f\to f^{(\alpha)}$ in $L_1(\mathbb{R})$ as $h\to 0+$.
For $p=0$, $\alpha>0$ and $f\in X_{\alpha}(\mathbb{R}_+)$, see \cite[Theorem 13]{Westphal1974} for the same result.
Theorem \ref{maintheorem} shows that for $f\in X_{\alpha+\beta}(\mathbb{R})$ the convergence rate in $L_1(\mathbb{R})$ is of order $h^\beta$, $0<\beta\le 1$ as $h\to 0+$, and if $p\le 0$, then the same holds for $f\in X_{\alpha+\beta}(\mathbb{R}_+)$ in $L_1(\mathbb{R}_+)$.

Furthermore, a detailed error analysis allows for higher order approximations by combining Gr\"unwald formulas with different shifts $p$ and accuracy $h$, cancelling out higher order terms. This will be shown in Corollary \ref{HigherOrderConsistency}.

In the error analysis of Theorem \ref{maintheorem} below, the function
\begin{equation}
\label{omega}
\omega_{p,\alpha} (z) = \left ( \frac {1-e^{-z}}{z} \right) ^ \alpha e^{zp},
\end{equation}
where $\alpha \in \mathbb{R}_+ $ and $p\in \R$, plays a crucial role. As we take the negative real axis as the branch cut for the fractional power, $\omega_{p,\alpha}$ is analytic, except where $(1-e^{-z})/z$ is on the negative real axis and at $z=0$. As $\lim_{z\to 0}\omega_{p,\alpha} (z)=1$, the singularity at zero is removable and hence there exists $R>0, a_j\in\R$ such that
\begin{equation}
\label{omegataylor}
\omega_{p,\alpha} (z) = \sum_{j=0}^\infty a_j z^j \; \text {for all}\; |z|<2R.
\end{equation}
 In particular, $a_0=1$ and $a_1=(p-\alpha/2)$. Furthermore, this implies that there exists $C>0$ such that
\begin{equation}
\label{boundforomegaminusone}
|\omega_{p,\alpha} (z) -1| \leq \left\{ \begin{array}{cl}
C |z| & \textnormal {for} \; |z| < R, \\
C & \textnormal {for} \: z\in i\R,\\
C  & \textnormal {for} \: \mathrm{Re} z\ge 0 ~\& ~p\le 0.\\
\end{array}\right.
\end{equation}
It was shown in \cite[Lemma 2]{Westphal1974} that
\begin{equation}\label{q_alpha}
\left(\frac{1-e^{-z}}{z}\right)^{\alpha}=\int_0^{\infty}e^{-zt}q_{\alpha}(t)\,\dd t,~\Re z\ge 0,~q_\alpha\in L_1(\mathbb{R}_+),
\end{equation}
and $\int_0^\infty q_\alpha(t)\,\dd t=1$.\\
Moreover, for $\alpha >1 $ and $k \in \mathbb{R},$ or for $0< \alpha <1 $ and $-\pi  < k < \pi,$ it is easily verified that
\begin{align}
\label{boundforderivativeomega}
\left|\frac{d}{dk}(\omega_{p,\alpha} (-ik))\right| \leq C.
\end{align}

\begin{lemma}
\label{ghatandghatprimeinl2}
Let $$\hat g_{\beta,N,p}(k)=\frac{\omega_{p,\alpha}(-ik)-\sum_{j=0}^N a_j (-ik)^j}{(-ik)^{N+\beta}},$$ where $\omega_{p,\alpha}(z)$ is given by (\ref{omega}), $a_j$ by \eqref{omegataylor} and $0<\beta\le1$, $N\in\mathbb N$. Then $\hat g_{\beta,N,p}$ is the Fourier transform of some $g_{\beta,N,p} \in L_1$. Furthermore, $
\supp(g_{\beta,N,p})\subset [-p,\infty)\cup\R_+$.
\end{lemma}
\begin{proof} Note that for $N,p=0$, $\beta<1$, this was shown in \cite[Lemma 5]{Westphal1974}. Let $N=0$, $\beta<1$ and $p\neq 0$. Then
$$\hat g_{\beta,0,p}(k)=\frac{\omega_{p,\alpha}(-ik)-1}{(-ik)^{\beta}}= e^{-ikp}g_{\beta,0,0}(k)+\frac{e^{-ikp}-1}{(-ik)^{\beta}}.$$
The inverse of the first term is given by $t\mapsto g_{\beta,0,0}(t+p)$; the second term inverts to $t\mapsto ((t+p)^{\beta-1}H(t+p)-t^{\beta-1}H(t))/\Gamma(\beta)$,  where $H$ is the unit step function. Clearly, both functions are in $L_1$ and their support is contained in $[-p,\infty)\cup\R_+$.
Assume $N=0$ and $\beta=1$, or $N\ge 1$.

Case 1: $\alpha\ge 1$ in equation \eqref{omega}. Clearly, each term of $\hat g$ is the Fourier transform of a tempered distribution satisfying the support condition and hence all that remains to show is that $\hat g_{\beta,N,p}$ is the Fourier transform of an $L_1$ function. Note that $\hat g_{\beta,N,p}\not \in L_2$ for $\beta<1/2$. Hence we use Corollary \ref{partUnity}, choosing a particular partition of unity $\left(\phi_j\right)_{j\in\mathbb Z}$.
If in \eqref{omegataylor} $R>1$, set $R=1$ and define $\phi_0(x)=1$ for $|x|<R/2$, $2(R-|x|)/R$ for $R/2\le|x|\le R$, $0$ else. Then $\supp(\phi_0)\subset[-R,R]$. Define $\phi_1(x)=2(x-R/2)/R$ for $R/2\le|x|\le R$, $\phi_1(x)=1$ for $R\le x\le 1$ and $\phi_1(x)=2-x$ for $1<x<2$ and   $\supp(\phi_1)\subset[R/2,2]$.
For $j>1$, let
\begin{equation*}
\phi_j(x)=\left\{ \begin{array}{cl}
(x-2^{j-2})/2^{j-2} & \textnormal {for} \; 2^{j-2}<x<2^{j-1}, \\
(2^j-x)/2^{j-1} & \textnormal {for} \: 2^{j-1}<x<2^{j},\\
0  & \mathrm{else}.\\
\end{array}\right.
\end{equation*}
Then $\supp(\phi_j)\subset [2^{j-2},2^{j}]$. For $j<0$, define $\phi_{j}(x)=\phi_{-j}(-x)$. Hence there exists $C$ such that $\|\phi'_j(x)\|_\infty\le C2^{-|j|+2}$ for all $j\in\mathbb Z$. By Corollary \ref{partUnity}, we are done if we can show that $\sum_j \|\hat g\phi_j\|_{L_2}^{\frac12}\|(\hat g\phi_j)'\|_{L_2}^{\frac12}<\infty.$

As $\hat g$ is analytic in $(-2R,2R)$, $\hat g\phi_0\in W_{2}^1(\R)$. For $j\neq 0$,
\begin{equation*}\begin{split}\hat g_{\beta,N,p}(k)\phi_j(k)=&\frac{\omega_{p,\alpha}(-ik)-\sum_{m=0}^{N-1} a_m (-ik)^m}{(-ik)^{N+\beta}}\phi_j(k)- \frac{ a_N}{(-ik)^{\beta}}\phi_j(k)\\=&T_1(k)-T_2(k).\end{split}\end{equation*}

By \eqref{boundforomegaminusone} and \eqref{boundforderivativeomega}, $\omega_{p,\alpha}$ and $\omega'_{p,\alpha}$ are bounded. Recall that either $N=0$ and $\beta=1$, or $N\ge 1$. In either case, the exponent in the denominator of $T_1$ is at least one. Hence there exists $C$ independent of $j$ such that
$|T_1(k)|\le C 2^{2-|j|}$, and $|T_1'(k)|\le C 2^{2-|j|}$.
The length of the support of $T_1$ is less than $2^{|j|}$. Hence
$$\|T_1\|_{L_2}^{\frac12}\le C^{1/2}(2^{4-2|j|}2^{|j|})^{1/4}\le C 2^{-|j|/4}.$$
Same for $\|T_1'\|_{L_2}^{\frac12}$. To estimate $T'_2$ recall that $\|\phi_j'\|_\infty\le 2^{-|j|+2}$. Thus there exists $C$ independent of $j$ such that
$|T'_2(k)|\le C2^{(|j|-2)(-\beta-1)}+ C2^{(|j|-2)(-\beta)}2^{-|j|+2}=C2^{3+\beta}2^{-(\beta+1)|j|},$ and hence
$$\|T'_2\|_{L_2}^{\frac12}\le C(2^{-2(\beta+1)|j|}2^{|j|})^{1/4}=C2^{-|j|(2\beta+1)/4}.$$
The norm of $T_2$ is bounded by
$$\|T_2\|_{L_2}^{\frac12}\le C (2^{-2|j|\beta}2^{|j|})^{1/4}\le C 2^{|j|(1-2\beta)/4}. $$
This implies that
$\|\hat g_{\beta,N,p}\phi_j\|_{L_2}\|(\hat g_{\beta,N,p}\phi_j)'\|_{L_2}\le C 2^{|j|(1-2\beta)/4}2^{-|j|/4}=C2^{-|j|\beta/2},$ and hence $\sum_j \|\hat g_{\beta,N,p}\phi_j\|_{L_2}\|(\hat g_{\beta,N,p}\phi_j)'\|_{L_2}< \infty$ and therefore $g_{\beta,N,p}\in L_1(\R)$.

Case 2: $\alpha<1$ in equation \eqref{omega}. For $\alpha<1$, $\omega_{p,\alpha}'$ is not bounded; if $\alpha\le 1/2$ it is not even locally in $L_2$, making the above method unfeasible. Instead we use induction on $N$.  We already established that the assertion holds for $N=0$, $\beta<1$, so assume it holds for some $N$; i.e. assume $\hat g_{\beta,N,p}$ is the Fourier transform of an $L_1$ function satisfying the support condition. Then using the fact that the convolution of $L_1$ functions is an $L_1$ function and the fact that we established the assertion for $\alpha =1,$ we obtain that $k\mapsto \hat g_{\beta,N,p}(k)\frac{\frac{1-e^{ik}}{-ik}-1}{-ik}$ is also the Fourier of an $L_1$ function satisfying the support condition as the support of the inverse of the second factor is in $\R_+$. Furthermore,
\begin{equation*}\begin{split}
\hat g_{\beta,N,p}(k)\frac{\frac{1-e^{ik}}{-ik}-1}{-ik}=&\frac{\left ( \frac {1-e^{ik}}{-ik} \right) ^ {\alpha+1} e^{-ikp}- \left(\frac{1-e^{ik}}{-ik}\right)\sum_{j=0}^N a_j(-ik)^j}{(-ik)^{N+1+\beta}}\\
&- \frac{\left ( \frac {1-e^{ik}}{-ik} \right) ^ {\alpha} e^{-ikp}- \sum_{j=0}^N a_j(-ik)^j}{(-ik)^{N+1+\beta}}\\
=&\frac{\left ( \frac {1-e^{ik}}{-ik} \right) ^ {\alpha+1} e^{-ikp}- \sum_{j=0}^{N+1} b_j(-ik)^j}{(-ik)^{N+1+\beta}}\\
&-\sum_{j=0}^N a_j\frac{ \left(\frac {1-e^{ik}}{-ik}\right)-\sum_{m=0}^{N+1-j}c_m (-ik)^m}{(-ik)^{N+1-j+\beta}}\\
&-\frac{\left ( \frac {1-e^{ik}}{-ik} \right) ^ {\alpha} e^{-ikp}- \sum_{j=0}^{N+1} d_j(-ik)^j}{(-ik)^{N+1+\beta}}\\
=&I_1-I_2-I_3,
\end{split}\end{equation*}
where the $b_j$ are Taylor coefficients of $\omega_{p,\alpha+1} $, the $c_m$ are the Taylor coefficients of $\frac{1-e^{-z}}{z}$ and $d_j$ are such that equality holds. As by case 1, $I_1$ and $I_2$ are Fourier transforms of $L_1$ functions, so is $I_3$. Hence $I_3$ at $k=0$ has to be bounded and therefore $d_j=a_j$ and hence $\hat g_{\beta,N+1,p}=I_3$.

The same argument about the convolution of $L_1$-functions applies to the remaining case of $N=0,\beta=1$, using \eqref{q_alpha} and the fact that
$$\omega_{p,\alpha}(-ik)\frac{\frac{1-e^{ik}}{-ik}-1}{-ik}= \frac{\omega_{p,\alpha+1}(-ik)-1}{-ik}-\frac{\omega_{p,\alpha}(-ik)-1}{-ik}.$$
\end{proof}

\begin{theorem}
\label{maintheorem}
Let $0\le\beta\le 1$. Then there exists $C>0$ such that
$f\in X_{\alpha+\beta}(\mathbb{R})$ implies that $$\left\| A^\alpha _{h,p} f - f^{(\alpha)} \right\|_{L_1(\mathbb{R})} \leq C h^\beta \left\|f\right\|_{\alpha +\beta}$$ as $h \to 0+.$ If $p\le 0$ and $f\in X_{\alpha+\beta}(\mathbb{R}_+)$,  then $\left\| A^\alpha _{h,p} f - f^{(\alpha)} \right\|_{L_1(\mathbb{R}_+)} \leq C h^\beta \left\|f\right\|_{\alpha +\beta}$. In case of $\beta=0$, $f\in X_\alpha$, $\left\| A^\alpha _{h,p} f - f^{(\alpha)} \right\|_{L_1(\mathbb{R})}\to 0$.
\end{theorem}
\begin{proof}
The case $\beta=0$ was already shown in \cite[Proposition 4.9]{Baeumer2009a} and is only included in the statement of the theorem for completeness.

To prove the first statement, first note that $(1+z)^\alpha =\sum_{m=0}^\infty {\binom{\alpha}{m} } z^m$ for $z \in \mathbb{C}, \; \left|z\right| < 1$ and $\alpha \in \mathbb{R}_{+}.$ Moreover, the binomial coefficients are related to the Gamma function by the following equation $ {\binom{\alpha}{m}}= \frac{(-1)^m \Gamma (m- \alpha)}{\Gamma (-\alpha) \Gamma(m+1)}\;.$
Taking Fourier transforms in (\ref{grunwald}) we get
\begin{align}
\label{fouriertransformomega}
\widehat{(A_h^\alpha f)} (k)  = &  h^{-\alpha} \sum_{m=0}^\infty (-1)^m {\binom{\alpha}{m}} e^{ik(m-p)h} \hat{f} (k)
						    		 =  h^{-\alpha} e^{-ikhp} (1-e^{ikh})^\alpha  \hat{f} (k) \nonumber\\
										 = &  (-ik)^\alpha \omega_{p,\alpha} (-ikh) \hat {f} (k)
										 = \widehat{f^{(\alpha)}} (k)+ \hat {\zeta}_h (k),\nonumber
 \end{align}
where $\omega_{p,\alpha}(z)$ is given by (\ref{omega}) and $\hat {\zeta}_h (k) = (-ik)^\alpha (\omega_{p,\alpha} (-ikh) -1) \hat {f} (k).$
We rewrite the error term $\hat{\zeta}_h (k)$ as a product of Fourier transforms of $L_1$-functions:
\begin{equation}
\label{errorasconvolutionandproduct}
\hat{\zeta}_h (k)	= (-ik)^\alpha (\omega_{p,\alpha} (-ikh) -1) \hat {f} (k)
									=  h^\beta \hat{g}_{\beta,0,p} (kh) \widehat {f^{(\alpha +\beta)}} (k),
\end{equation}
where, by Lemma \ref{ghatandghatprimeinl2},  $\hat{g}_{\beta,0,p} (k) = \left( \frac {\omega_{p,\alpha} (-ik) -1}{(-ik)^\beta} \right)$ is the Fourier transform of an $L_1$ function $g_{\beta,0,p}$ and $f^{(\alpha +\beta)}\in L_1$
by assumption. Thus $\hat{\zeta}_h$ is indeed the Fourier transform of an $L_1$-function $\zeta_h$ and, by \eqref{normdensity}, \eqref{multipliernormequivalence} and Lemma \ref{ghatandghatprimeinl2},
\begin{align*}
\left\|\zeta_h\right\|_{L_1(\mathbb{R})}  &\leq  h^\beta\|T_{\hat{g}_{\beta,0,p}(h\cdot)}\|_{\mathcal{B}(L_1(\mathbb{R}))}\left\|f^{(\alpha +\beta)}\right\|_{L_1(\mathbb{R})}=h^\beta\left\|{g}_{\beta,0,p} \right\|_{L_1(\mathbb{R})}\left\|f^{(\alpha +\beta)}\right\|_{L_1(\mathbb{R})}\\
															&\leq Ch^\beta\left\|f\right\|_{\alpha +\beta}.
\end{align*}
Finally, the second statement follows the same way, taking Remark \ref{rem:l1r+} into account and noting that if $p\le 0$, then $\supp (\zeta_h)\subset \mathbb{R}_+$ and using Lemma \ref{ghatandghatprimeinl2} in \eqref{errorasconvolutionandproduct}.
\end{proof}

\begin{remark}\label{rem:grumu}
Let $\alpha \in \mathbb{R}_+, \; 2q-1< \alpha <2q+1, \; q \in \mathbb{N},$ and
\begin{equation}
\label{grunwaldmultgeneral}
\psi_{\alpha,h,p}(z)=(-1)^{q+1}h^{-\alpha} e^{-hpz} (1-e^{hz})^\alpha .
\end{equation}
Then the proof of Theorem \ref{maintheorem} shows that the Gr\"unwald formula \eqref{grunwald} can be expressed in the multiplier notation of Section \ref{sec:pre} as
$
A_{h,p}^{\alpha}=T_{(-1)^{q+1}\psi^p_h}
$
where
\begin{equation}
\label{grunwaldmultipliersymbol}
\psi^p_h(k)=\psi_{\alpha,h,p}(ik)=(-1)^{q+1}h^{-\alpha} e^{-ikhp} (1-e^{ikh})^\alpha .
\end{equation}
\end{remark}

Analyzing the error term $\hat \zeta_h$ in \eqref{errorasconvolutionandproduct} further  allows for higher order approximations by combining Gr\"unwald formulae with different shifts $p$ and accuracy $h$, cancelling out lower order error terms. Consider the Taylor expansion of $\omega_{p,\alpha}$ in \eqref{omegataylor} and let $N\ge 0$ and for $0\le j\le N$ let $b_j,p_j\in\R$ and $c_j>0$ be such that there exist $d_j$ with
\begin{equation*}\label{hogs}
\sum_{j=0}^{N} b_j \omega_{p_j,\alpha}(c_jz)-1=\sum_{j=N+1}^\infty d_j z^j
\end{equation*}
for $|z|<2R$; i.e. consider a linear combinations of $\omega_{p_j,\alpha}$'s cancelling out the lower order terms. Define
\begin{equation}\label{higherOrderA}
\tilde A^\alpha_hf:=\sum_{j=0}^{N} b_j A_{c_jh,p_j}^\alpha f=T_{(-1)^{n+1}\sum_{j=0}^{N} b_j\psi^{p_j}_{c_jh}}
\end{equation}
to be an $N+1$ order Gr\"unwald approximation. This is justified according to the following Corollary.

\begin{corollary}\label{HigherOrderConsistency}
  Let $0<\beta\le 1$, $N\in\mathbb N$ and $\tilde A^\alpha_h$ be an $N+1$ order Gr\"unwald approximation. Then there exists $C>0$ such that $f\in X_{\alpha+N+\beta}(\R)$ implies that
  $$\left\|\tilde A_h^\alpha f-f^{(\alpha)}\right\|_{L_1(\R)}\le Ch^{N+\beta}\|f\|_{\alpha+N+\beta}.$$
  as $h\to 0^+$. If $p_j\le 0$ for all $0\le j\le N$ and $f\in X_{\alpha+N+\beta}(\R_+)$ then
  $$\left\|\tilde A_h^\alpha f-f^{(\alpha)}\right\|_{L_1(\R_+)}\le Ch^{N+\beta}\|f\|_{\alpha+N+\beta}.$$
\end{corollary}

\begin{proof} Note that $\sum_{j=0}^N b_j=1$ and for $n>0$, we have that $\sum_{j=0}^N b_j a_{j,n}c_j^n=0$.
  Following the same argument that led to \eqref{errorasconvolutionandproduct}, we obtain that the error term can be expressed as
  \begin{equation*}\begin{split}\hat \zeta_h(k)=&\widehat{\left(\tilde A_h^\alpha f-f^{(\alpha)}\right)}(k)=\sum_{j=0}^N b_j\left(\omega_{p_j,\alpha}(-ikc_jh)-1\right)(-ik)^\alpha\hat f(k)\\
  =&\sum_{j=0}^N b_j\left(\omega_{p_j,\alpha}(-ikc_jh)-\sum_{n=0}^N a_{j,n}(-ikc_jh)^n\right)(-ik)^\alpha\hat f(k)\\
  =&h^{N+\beta}\sum_{j=0}^N b_jc_j^{N+\beta}\frac{\omega_{p_j,\alpha}(-ikc_jh)-\sum_{n=0}^N a_{j,n}(-ikc_jh)^n}{(-ikc_jh)^{N+\beta}}(-ik)^{\alpha+N+\beta}\hat f(k),
  \end{split}
  \end{equation*}
  where $a_{j,n}$ are the Taylor coefficients of $\omega_{p_j,\alpha}$. By Lemma \ref{ghatandghatprimeinl2}, $$\hat g(k)=\sum_{j=0}^N b_jc_j^{N+\beta}\frac{\omega_{p_j,\alpha}(-ikc_j)-\sum_{n=0}^N a_{j,n}(-ikc_j)^n}{(-ikc_j)^{N+\beta}}$$ is the finite sum of Fourier transforms of $L_1$ functions and hence $\hat \zeta_h$ is the Fourier transform of an $L_1$ function with
  $$\| \zeta_h\|_{L_1(\R)}\le Ch^{N+\beta}\|f\|_{\alpha+N+\beta}.$$
  The second statement follows along the same lines.
\end{proof}

\section{Stability and Smoothing: Semigroups generated by periodic multipliers approximating the fractional derivative}

The next result is the main technical tool of the paper. It gives a sufficient condition for multipliers associated with difference schemes approximating the fractional derivative to lead to stable schemes with desirable smoothing. In error estimates  later, the smoothing of the schemes will be used in an essential way to reduce the regularity requirements on the initial data to obtain optimal convergence rates when considering space-time discretizations of Cauchy problems with fractional derivatives or, more generally, fractional powers of operators.
\begin{theorem}
\label{generalresult}
Let $\alpha \in \mathbb{R}_+ $ and $\psi$ be an absolutely continuous $2\pi$-periodic function that satisfies the following:
\begin{enumerate}
\item[(i)] $\left| \psi (k) \right| \leq C \left|k \right|^\alpha \text{for some} \; C>0,$
\item[(ii)] $\left| \psi' (k) \right| \leq C' \left|k \right|^{\alpha -1}\text{for some} \; C'>0,$
\item[(iii)] $\Re(\psi(k)) \leq -c \left|k \right|^\alpha \text{for some} \; c>0.$
\end{enumerate}
Then $\psi \in W_{r,per}^1[-\pi,\pi],$ where $r=2$ if $\alpha >\frac{1}{2}$ and $r < \frac{1}{1-\alpha}$ if $ \alpha \le \frac{1}{2}.$ Moreover,
\begin{enumerate}
\item[(a)] $\left\|T_{e^{t\psi}}\right\|_{\mathcal{B} (L_1)} \leq K \;\text{for}\; t\geq 0,$
\item[(b)] $\left\|T_{\psi e^{t\psi}}\right\|_{\mathcal{B} (L_1)} \leq \frac{M}{t} \;\text{for}\; t> 0,$
\end{enumerate}
where $K$ and $M$ depend on $c,C \; \text{and}\; C'$ above.
\end{theorem}
\begin{proof}
Set $r=2$ if $\alpha >\frac{1}{2}$ and $r < \frac{1}{1-\alpha}$ if $ \alpha \le \frac{1}{2},$ where
\begin{equation}
\label{holderconjugates}
1< r \leq 2 \quad \text{and}\quad \frac{1}{r}+\frac{1}{s}=1.
\end{equation}
Note that $r(\alpha -1) >-1,$ then from the assumptions we have $\psi \in W_{r,per}^1[-\pi,\pi],$ and for $t\ge 0,$ we also have $e^{t \psi} \in W_{r,per}^1[-\pi,\pi].$ By Theorem \ref{thm:eta},
\begin{equation}
\label{carlsonbeurlingforetpsigeneral}
\left\|T_{e^{t \psi }}\right\|_{\mathcal{B} (L_1)}\leq \left|a_0 \right| + C \left\| e^{t \psi }\right\|^{\frac{1}{s}}_{L_r}\left\|(e^{ t \psi })'\right\|^{\frac{1}{r}}_{L_r}.
\end{equation}
Firstly $\left|a_0 \right|= \left|\frac{1}{2\pi}\int_{-\pi}^\pi e^{t \psi (k)} dk \right| \le \frac{1}{2\pi} \int_{-\pi}^\pi e^{t \Re (\psi (k))} \,\dd k \le \frac{1}{2\pi} \int_{-\pi}^\pi e^{-ct |k|^\alpha} \,\dd k \le 1,$
where we have used Assumption (iii). Now, using Assumption (iii) again together with the substitution $\tau = t^{\frac{1}{\alpha}}|k|,$
\begin{align}
\label{etpsiestimategeneral}
\left\|e^ {t \psi}\right\|_{L_r[-\pi,\pi]} ^{\frac{1}{s}}= & \left(\int_{-\pi} ^ \pi \left | e^{t \psi (k)} \right |^r dk \right)^ \frac{1}{rs}
						 	\le C \left( \int_{-\pi}^\pi e^{rt \Re (\psi (k))} \,\dd k \right)^\frac{1}{rs} \nonumber\\
						 	\le & C \left( \int_{-\pi}^\pi e^{-rct |k|^\alpha} \,\dd k \right) ^{\frac{1}{rs}}
							\le C \left( \frac{1}{t^{\frac{1}{\alpha}}}\int_{\mathbb{R}} e^{-rc|\tau|^\alpha} \,\dd\tau \right)^{\frac{1}{rs}}
							\le C t^{\frac{-1}{\alpha rs}}.
\end{align}
Making use of Assumptions (ii) and (iii), we have
\begin{equation*}
\label{e2}
\left | \frac{d}{dk}\left(e^{ t \psi(k) }\right)\right |^r = \left| t \frac{d\psi (k)}{dk} e^{ t \psi (k)}\right|^r
									\leq C \left(t^r \left|k\right|^{r (\alpha -1)}\right)e^{-rct\left|k\right|^\alpha},
\end{equation*}
Since $r (\alpha -1) >-1,$ an application of \eqref{holderconjugates} and the substitution $\tau = t^{\frac{1}{\alpha}}|k|,$ yields
\begin{align}
\label{derivativeetpsiestimategeneral}
\left\|\left(e^{ t \psi  }\right)'\right\|_{L_r[-\pi,\pi]}^{\frac{1}{r}}
																			\leq & C \left (t^{\frac{r-1}{\alpha}} \int_{\mathbb{R}} \tau^{r (\alpha -1)} e^{-rc \tau^ \alpha } d \tau \right ) ^{\frac{1}{r^2}}
																			\leq  C  t^\frac{1}{\alpha rs}.
\end{align}
The proof of (a) is complete in view of \eqref{carlsonbeurlingforetpsigeneral}, \eqref{etpsiestimategeneral} and \eqref{derivativeetpsiestimategeneral}.

We have that $\psi e^{t\psi}\in W_{r,per}^1[-\pi,\pi]$ and once again by Theorem \ref{thm:eta},
\begin{equation}
\label{cpsigeneral}
\left\|T_{\psi e^{t \psi }}\right\|_{\mathcal{B} (L_1)} \leq \left|a_0 \right|+ C \left\|\psi e^{t \psi }\right\|^{\frac{1}{s}}_{L_r[-\pi,\pi]}\left\|\left( \psi e^{ t \psi }\right)'\right\|^{\frac{1}{r}}_{L_r[-\pi,\pi]}.
\end{equation}
Note that $\left|a_0 \right| \le \frac{1}{2\pi} \int_{-\pi}^\pi \left|\psi(k)\right|e^{t \Re (\psi (k))} \,\dd k \le C \int_{-\pi}^\pi \left|k\right|^\alpha e^{-ct |k|^\alpha} \,\dd k \le C/t,$
where we have used Assumptions (i) and (iii). The use of the substitution $\tau = t^{\frac{1}{\alpha}}|k|$ and the Assumptions (i) and (iii), yield
\begin{equation}\label{dpgeneral}
\left\| \psi e^{t \psi }\right\|^{\frac{1}{s}}_{L_r[-\pi,\pi]}\le C\left( \int_{\mathbb{R}}|k|^{r\alpha}e^{-rct|k|^{\alpha}}\,\dd k \right)^{\frac{1}{rs}}\le Ct^{-\frac{1}{s}-\frac{1}{\alpha r s}}.
\end{equation}
We also have by virtue of \eqref{derivativeetpsiestimategeneral} and the three assumptions,
\begin{align*}
\left|\frac{d}{dk}(\psi e^{ t \psi })(k)\right|^r
& =\left| \psi(k) \frac{d}{dk}\left(e^{ t \psi (k) }\right) + \frac{d\psi(k)}{dk}  e^{ t \psi (k)}\right|^r\\
& \le 2^{r-1}\left(t^r\left|\psi(k)\right|^r + 1\right)\left|\left(\psi(k)\right)' \right|^r \left| e^{ t \psi (k) } \right|^r \\
&\le C(|k|^{r(\alpha-1)}+t^r |k|^{r(2\alpha -1)})e^{-rct|k|^{\alpha}}.
\end{align*}
Thus, using \eqref{holderconjugates} and the substitution $\tau = t^{\frac{1}{\alpha}}|k|,$ and noting that $r(2\alpha -1) > r(\alpha -1) >-1,$
\begin{align}
\label{derivativedpgeneral}
\left\|\frac{d}{dk}(\psi e^{ t \psi })\right\|_{L_r[-\pi,\pi]}^{\frac{1}{r}}
\le & C t^{-\frac{1}{r}+\frac{1}{\alpha r s}}\left( \int_{\mathbb{R}}(|\tau|^{r(\alpha-1)}+|\tau|^{r(2\alpha -1)})e^{-rc \tau^{\alpha}}\,\dd \tau \right) ^{\frac{1}{r^2}}\nonumber \\
\leq & C t^{-\frac{1}{r}+\frac{1}{\alpha r s}},
\end{align}
and the proof of (b) is complete in view of \eqref{holderconjugates}, \eqref{cpsigeneral}, \eqref{dpgeneral} and \eqref{derivativedpgeneral}.
\end{proof}

\subsection{The shifted Gr\"{u}nwald formula of order $1$}
\label{spectrumofshiftedgrunwald}

First, we make the following two observations about the multiplier associated with the shifted Gr\"{u}nwald formula given by \eqref{grunwaldmultipliersymbol}. Note that for $p,k \in \mathbb{R}$ and $h, \alpha \in \mathbb{R}_+, $ such that $2q-1 < \alpha < 2q+1,\; q \in \mathbb{N},$
\begin{equation}
\label{psiintermsofomega}
\psi^p_h(k)=(-1)^{q+1}h^{-\alpha} e^{-ikhp} (1-e^{ikh})^\alpha = (-1)^{q+1} h^{-\alpha}(-ikh)^\alpha \omega_{p,\alpha}(-ikh),
\end{equation}
where $\omega_{p,\alpha}(z)$ is given by \eqref{omega}. Note also that
\begin{equation}
\label{scalingpsi}
\psi^p_h(k)=h^{-\alpha} \psi^p_1(hk).
\end{equation}
We now show that the range of the symbol associated with the shifted Gr\"{u}nwald formula is completely contained in a half-plane if and only if the shift is optimal.

\begin{prop}
\label{thm:spg}
Let
$\psi^p_h (k)= (-1)^{q+1} h^{-\alpha}e^{-iphk}(1-e^{ikh})^\alpha,$
where the shift $p\in \mathbb{N}$, $k \in \mathbb{R}$, $h>0$, $\alpha \in \mathbb{R}_+$ and $q \in \mathbb{N}$ such that $2q-1 < \alpha < 2q+1$.  Then
\begin{enumerate}
\item[(a)] $\psi^p_h$ satisfies Assumptions (i) and (ii) of Theorem \ref{generalresult} with $C,\;C'$ independent of $h.$
\item[(b)] $\Re (\psi^p_h)$ does not change sign if and only if $\left|p- \frac{\alpha}{2} \right| < \frac{1}{2}$ if and only if $\psi^p_h$ satisfies the Assumption (iii) of Theorem \ref{generalresult} with $c$ independent of $h.$
\end{enumerate}
\end{prop}
\begin{proof}
Proof of (a):

Recall that by \eqref{boundforomegaminusone} and \eqref{boundforderivativeomega}, $\left| \omega_{p,\alpha}(-ikh)\right| \leq C$ and $\left| \omega'_{p,\alpha}(-ikh)\right| \leq C,$ for $k \in [-\pi,\pi],$ respectively. Thus using \eqref{psiintermsofomega}, for $k \in [-\pi,\pi],$
\begin{equation}
\label{modpsishiftedgrunwald}
\left|\psi^p_h (k)\right|\leq C \left|k\right|^\alpha, \text{and}
\end{equation}
\begin{equation}\label{derivativepsishiftedgrunwald}
\left| \frac{d \psi^p_h (k)}{dk} \right| =h^{-\alpha}\left|-\alpha ih (-ikh)^{\alpha -1} \omega_{p,\alpha} (-ikh) + ( -ikh)^\alpha \omega'_ {p,\alpha}(-ikh) \right | \leq  C|k|^{\alpha -1},
\end{equation}
for some $C>0.$

Proof of (b): It is enough to consider the case $h=1$ in view of \eqref{scalingpsi}, so let $\psi:=\psi^p_1$.
Since $\psi$ is $2 \pi$-periodic and $\psi(k)=\overline{\psi(-k)}$, it is sufficient to consider $k\in [0, \pi].$
Now
\begin{equation*}
\psi (k)				 = (-1)^{q+1} e^{-ipk}\left(e^{\frac{ik}{2}}(e^{\frac{-ik}{2}}-e^{\frac{ik}{2}})\right)^\alpha
								 = (-1)^{q+1} e^{i(\frac{\alpha}{2}-p)k}\left(-2i\sin \left(\frac{k}{2}\right)\right)^\alpha.
\end{equation*}
Using the fact that for $x\ge 0$, $(-ix)^\alpha=x^\alpha e^{-i\alpha\frac\pi 2}$,
we obtain
\begin{equation*}
\label{symbolexplicit}
\psi(k)=(-1)^{q+1} 2^\alpha  \sin^\alpha \left( \frac{k}{2} \right) e^{i(\frac{\alpha}{2}-p)k - \frac{i\alpha \pi}{2}};\;\;\; 0\le k\le \pi,
\end{equation*}
and therefore
\begin{equation}
\label{realpartsymbol}\begin{split}
\Re \left( \psi (k)\right) =& (-1)^{q+1} 2^\alpha  \sin^\alpha \left( \frac{k}{2} \right) \cos\left( (\frac{\alpha}{2}-p)k - \frac{\alpha \pi}{2}\right)\\
=&(-1)^{q+1-p} 2^\alpha  \sin^\alpha \left( \frac{k}{2} \right) \cos\left( (\frac{\alpha}{2}-p)(k - \pi)\right).
\end{split}
\end{equation}

Clearly, as $0\le k\le \pi$, \eqref{realpartsymbol} changes sign if and only if $\left| \frac{\alpha}{2}-p \right| > \frac{1}{2}$. Note that by assumption $|\frac{\alpha}{ 2}- p|\neq\frac12$. Furthermore, Assumption (iii) of Theorem \ref{generalresult} implies that there is no sign change. Hence all that remains to show is that $\left|p -\frac{\alpha}{2}\right| < \frac{1}{2}$ implies Assumption (iii) of Theorem \ref{generalresult}.  The fact that $c$ is independent of $h$ follows from \eqref{scalingpsi}.

Note that 
$\left| p - \frac{\alpha}{2}\right| < \frac{1}{2}$ implies that $p =q$ and hence, using that for $0\le x\le\pi$, $\sin(x/2)\ge x/\pi$ and $\cos((\frac{\alpha}{2}-p)(x-\pi))\ge \cos(-(\frac{\alpha}{2}-p)\pi)$ ,
\begin{equation}\label{reEst}\begin{split}\Re \left( \psi (k)\right) =&- 2^\alpha  \sin^\alpha \left( \frac{k}{2} \right) \cos\left( (\frac{\alpha}{2}-p)(k - \pi)\right)\\
\le&- 2^\alpha \left( \frac{k}{\pi} \right)^\alpha \cos\left( (\frac{\alpha}{2}-p)\pi\right) =- k^\alpha  2^\alpha\cos\left( (\frac{\alpha}{2}-p)\pi\right)/\pi^\alpha.
\end{split}\end{equation}
\end{proof}
Next, we show that with the optimal shift; i.e., $p=q$, the operators $T_{\psi_h}$ generate strongly continuous semigroups on $L_1(\mathbb{R})$, and in the case when $p=0$; i.e., $0<\alpha<1$, on $L_1(\mathbb{R}_+)$, that are bounded uniformly in $h$. Since the range of $\psi_h$ is always contained in the spectrum of $T_{\psi_h}$, the semigroups generated by $T_{\psi_h}$ will not be uniformly bounded if the shift is not optimal as shown by Proposition \ref{thm:spg}.
In fact we show more: if the shift is optimal then the semigroups are uniformly analytic in $h$; i.e., there is  $M> 0$ such that the uniform estimate $\|T_{\psi_he^{t\psi_h}}\|_{\mathcal{B} (L_1)}\le Mt^{-1}$ holds for $t,h>0$. This fact will have significance when proving error estimates for numerical schemes for fractional differential equations.

\begin{theorem}
\label{grunwaldscsganalsg}
Let $\alpha \in \mathbb{R}_+,\; 2p-1<\alpha<2p+1$, $p\in \mathbb{N},$ and
\begin{equation*}\label{psiop}
\psi_h(k)=(-1)^{p+1}h^{-\alpha}e^{-ipkh}\left( 1-e^{ikh}\right)^\alpha.
\end{equation*}
Then the following hold.
\begin{enumerate}
\item[(a)] $\{T_{e^{ t \psi_h}}\}_{t\ge 0}$ are strongly continuous semigroups on $L_1(\mathbb{R})$ that are bounded uniformly in $h>0$ and $t\ge 0$. In particular, if $1<\alpha<2$, then $\{T_{e^{ t \psi_h}}\}_{t\ge 0}$ is a positive contraction semigroup on $L_1(\mathbb{R})$ and for $0<\alpha<1$, on $L_1(\mathbb{R}_+).$
\item[(b)] The semigroups $\{T_{e^{ t \psi_h}}\}_{t\ge 0}$ are uniformly analytic in $h>0$; i.e. there is  $M> 0$ such that the uniform estimate $\|T_{\psi_he^{t\psi_h}}\|_{\mathcal{B} (L_1)}\le Mt^{-1}$ holds for $t,h>0$.
\end{enumerate}
\end{theorem}

\begin{proof}
Proof of (a): To begin, note that $\psi_h(k)= \psi^p_h(k)$ with $p=q.$
We only have to show that $\left\|T_{e^{ t \psi_h }}\right\|_{\mathcal {B} (L_1)}\le K $, for all $t\ge 0$ and $h>0$, for some $K\ge 1$ and strong continuity follows by \cite[Proposition 8.1.3]{olivebook}. Furthermore, it is enough to consider $h=1$ in view of \eqref{multipliernormequivalence} and \eqref{scalingpsi}, so let $\psi:=\psi_1$.
First, let $0<\alpha<1$ or $1<\alpha<2$ and hence $p=0$ or $1$, respectively. We have, taking Remark \ref{rem:grumu} into account, that
\begin{align*}
(T_{\psi}f)(x)&=(-1)^{p+1}(\sum_{m=0}^\infty (-1)^m {\binom{\alpha}{m}} f(x-(m-p)))\\
&= -{\binom{\alpha}{p}}f(x)+(-1)^{p+1}\sum_{m=0,m\neq p}^\infty (-1)^m {\binom{\alpha}{m}} f(x-(m-p))\\
&=(-{\binom{\alpha}{p}}I\,f)(x)+(T_{\tilde{\psi}}f)(x).
\end{align*}
Since $(-1)^{p+1}(-1)^m {\binom{\alpha}{m}}\ge 0$, for $m\neq p$, it follows that $T_{\tilde{\psi}}$
is a positive operator on $L_1(\mathbb{R})$ (or, $L_1(\mathbb{R}_+)$ for $0<\alpha<1$ by recalling Remark \ref{rem:l1r+}) and so is $e^{tT_{\tilde{\psi}}}=T_{e^{ t \tilde{\psi} }}$. Therefore, noting the fact that $\sum_{m=0}^\infty (-1)^m {\binom{\alpha}{m}}=0,$
$$
T_{e^{ t \psi_1 }}=e^{tT_{\psi}}=e^{t(-{\binom{\alpha}{p}}I+T_{\tilde{\psi}})}=e^{-{\binom{\alpha}{p}}t}e^{tT_{\tilde{\psi}}}\ge 0,
$$
and
$$
\|T_{e^{ t \psi}}\|_{L_1}\le e^{-{\binom{\alpha}{p}}t}e^{t\|T_{\tilde{\psi}}\|_{L_1}}=e^{-{\binom{\alpha}{p}}t}e^{(-1)^{p+1}\sum_{m=0,m\neq p}^\infty (-1)^m {\binom{\alpha}{m}}t}=1.
$$
Let now $\alpha>2,$ then by Proposition \ref{thm:spg}, $\psi_h$ satisfies the hypothesis of Theorem \ref{generalresult}
and the proof of (a) is complete.

Proof of (b): Let $\alpha >0,$ then the statement follows from Theorem \ref{generalresult} in view of Proposition \ref{thm:spg}.

\end{proof}

\subsection{Examples of second order stable Gr\"{u}nwald-type formulae}\label{subsec:2order}

Let $\alpha \in \mathbb{R}_+, \;2q-1< \alpha <2q+1,\; \text{and} \: q \in \mathbb{N}.$ Consider the mixture of symbols of Gr\"{u}nwald formulae yielding a second order approximation of the form
\begin{equation*}\label{hiorgru}
\phi_h (k):= a \psi^{p_1}_{h}(k) + (1-a) \psi^{p_2}_{2h} (k),
\end{equation*}
where $\psi_h^p (k)=(-1)^{q+1}h^{-\alpha}e^{-ipkh}(1-e^{ikh})^\alpha,$ the symbol of the $p$-shifted Gr\"{u}nwald formula and thus $\tilde A_h^\alpha=(-1)^{q+1}T_{\phi_h}$. There are of course many combinations of $\alpha, a, p_1$, and $p_2$ that yield a second order approximation; however only some are stable. It is straight-forward to show that if $0<\alpha<1$, then $a=2, p_1=p_2=0$ give a second order approximation and similarly for $1<\alpha<2$, $a=2-\frac{2}{\alpha},\;p_1=1,\;p_2=\frac{1}{2}$. So we only show stability.

\begin{prop}
\label{thm:mix}
Let $\phi_h (k)$ be as above, where $a=2,\;p_1=p_2=0,$ if $0<\alpha<1$ and $a=2-\frac{2}{\alpha},\;p_1=1,\;p_2=\frac{1}{2},$ if $1< \alpha <2.$  Then $\phi_h$ satisfies the assumptions of Theorem \ref{generalresult} with constants $c$, $C$ and $C'$ independent of $h$.
\end{prop}

\begin{proof}
Assumptions (i) and (ii) of Theorem \ref{generalresult} are clearly satisfied in view of \eqref{modpsishiftedgrunwald} and \eqref{derivativepsishiftedgrunwald}.

Proof of (iii):
As $$ \phi_h(k)=h^{-\alpha}\phi_1(hk),$$ it is sufficient to consider the case $h=1$, so let $\phi := \phi_1.$ That is,
\begin{equation*}
\phi (k)= \phi_{1} (k)	=(-1)^{q+1} \left(a e^{-ip_1 k}(1-e^{ik})^\alpha+(1-a) 2^{-\alpha}e^{-i{\alert 2}p_2 k}(1-e^{i2k})^\alpha\right)
\end{equation*}
and note that $\phi$ is $2 \pi$-periodic. By symmetry, we only need to consider the real part for $0 \leq  k \leq  \pi,$ and so using \eqref{realpartsymbol} and the double angle formula we have
\begin{equation*}
\label{rangemixed}
\Re(\phi(k))=  (-1)^{q+1} 2^{\alpha} \sin^\alpha \left(\frac{k}{2}\right) \left( a \cos A +(1-a)   \cos^\alpha \left(\frac{k}{2} \right) \cos B \right),
\end{equation*}
where $A=(\frac{\alpha}{2}-p_1)k - \frac{\alpha \pi}{2}$ and $B=(\alpha-2 p_2)k - \frac{\alpha \pi}{2}$.

We consider for $0\leq k \leq \pi,$ the function
$$F(k) = (-1)^{q+1} \left(a \cos A +(1-a)\cos B \right)$$
and show that $F(k)\le F(\pi)<0$. Then, since $a>0$ and $p_1$ is the optimal shift, if $\Re\left((1-a) \psi^{p_2}_h (k) \right)<0$, by \eqref{reEst}, $\Re(\phi(k))\le -a k^\alpha  2^\alpha\cos\left( (\frac{\alpha}{2}-p)\pi\right)/\pi^\alpha$. If $\Re\left((1-a) \psi^{p_2}_h (k) \right)>0$, we will have the estimate
$$\Re(\phi(k))\le  F(\pi) 2^{\alpha} \sin^\alpha \left(\frac{k}{2}\right)\le k^\alpha F(\pi)2^\alpha/\pi^\alpha.$$

An easy check shows that $F(\pi)<0$. It remains to show that $F'(k)>0$. Now, as in both cases $2p_1=p_2$ and $(1-a)(\alpha-2p_2)=-a \left(\frac\alpha 2-p_1\right)$,
\begin{equation*}\begin{split}
F'(k)=&(-1)^{q+1}\left(-a \left(\frac\alpha 2-p_1\right)\sin(A)-(1-a)\left(\alpha-2p_2\right)\sin(B)\right) \\= & C(\alpha)\left( \sin(A) -\sin(B) \right)= 2 C(\alpha) \cos\left( \frac{A+B}{2}\right) \sin\left(\frac{A-B}{2}\right)
\end{split}\end{equation*}
where $
C(\alpha)= -a(\frac{\alpha}{2}-p_1 )$, $\frac{A+B}{2}= (\frac{3 \alpha}{4}-p_1)k-\frac{\alpha \pi}{2}$, and $\frac{A-B}{2}=\frac{-\alpha k}{4}$. Checking the range of the arguments, the cosine factor is negative if $\alpha>1$ and positive if $\alpha<1$,  the sine factor is always negative, and hence $F'(k)>0$.
\end{proof}

\begin{theorem}\label{4.5}
Let $\phi_h (k)$ and $a$ be as in Proposition \ref{thm:mix}. If $0<\alpha<1$ and $a=2,$ or if $1<\alpha<2$ and $a=2-\frac{2}{\alpha},$ then $\{T_{e^{ t \phi_h}}\}_{t\ge 0}$ are semigroups on $L_1(\mathbb{R}_+),$ or $L_1(\mathbb{R}),$ respectively, that are uniformly bounded in $h$ and $t$, strongly continuous, and uniformly analytic in $h$.
\end{theorem}
\begin{proof}
The statement follows from Theorem \ref{generalresult} in view of Proposition \ref{thm:mix}.
\end{proof}

\section{Application to fractional powers of operators}

Let $X$ be a Banach space and $-A$ be the generator of a strongly continuous group of bounded linear operators $\{G(t)\}_{t\in \mathbb{R}}$ on $X$ with $\|G(t)\|_{\mathcal{B}(X)}\le M$ for all $t\in \mathbb{R}$ for some $M\ge 1$.
If $\mu$ is a bounded Borel measure on $\mathbb{R}$ and if we set $\psi(z):=\hat{\mu}(z)=\int_{-\infty}^{\infty}e^{zs}\,d\mu(s)$, ($z=ik$), we may define the bounded linear operator
\begin{equation}\label{hp}
\psi(-A)x:=\int_{\mathbb{R}}G(s)x\,\dd \mu(s), ~x\in X.
\end{equation}
It is well known that the map $\psi\to \psi(-A)$ is an algebra homomorphism and is called the Hille-Phillips functional calculus, see for example, \cite{Kant}. That is, if $\psi=\hat{\mu}$ and $\phi=\hat{\nu}$, for some bounded Borel measures $\mu$ and $\nu$, then $(\phi+\psi)(-A)=\phi(-A)+\psi(-A)$, $(\phi\cdot \psi)(-A)=\phi(-A)\psi(-A)$ and $(c\phi)(-A)=c\phi(-A)$, $c\in \mathbb{C}$. A simple transference principle shows, see e.g. \cite[Theorem 3.1]{Baeumer2009a}, that,
\begin{equation}\label{trans}
\|\psi(-A)\|_{\mathcal{B}(X)}\le M\|T_{k\mapsto\psi(ik)}\|_{\mathcal{B}(L_1(\mathbb{R}))}.
\end{equation}
Note that if $\supp \mu\subset \mathbb{R}_+$, then we may take $-A$ to be the generator of a strongly continuous semigroup and the properties of the Hille-Phillips functional calculus \eqref{hp} and the transference principle \eqref{trans} still holds.

Let $2p-1 < \alpha < 2p+1,\quad \alpha \in \mathbb{R}^+, ~p\in \mathbb{N}$ and $\{\mu_t\}_{t\ge 0}$ be the family of Borel measures on $\mathbb{R}$ such that $\hat{\mu_t}(z)=e^{t(-1)^{p+1}(-z)^{\alpha}}$, $(z=ik)$. Then the operator family given by
\begin{equation*}\label{stable}
S_{\alpha}(t)x:=\int_{\mathbb{R}}G(s)x\,\dd \mu_t(s), ~x\in X, t\ge 0,
\end{equation*}
is a uniformly bounded (analytic) semigroup of bounded linear operators on $X$, see \cite[Theorems 4.1 and 4.6]{Baeumer2009a} for the group case. In case $0<\alpha<1$, we have $\supp \mu_t\subset  \R^+$ and hence $-A$ is allowed to be a semigroup generator and $G$ to be a strongly continuous semigroup and the analyticity of $S_{\alpha}$ holds, see \cite{Bala} and \cite{Yos}. The fractional power $A^{\alpha}$ of $A$ is then defined to be the generator of $S_\alpha$ multiplied by $(-1)^{p+1}$. We note that the fractional power of $A$ may be defined via an unbounded functional calculus for group generators (or, semigroup generators), formally given by $f_{\alpha}(-A)$, where $f_{\alpha}(z)=(-z)^{\alpha}$. This coincides with the definition given here for groups and, in case $0<\alpha<1$, for semigroups, see \cite{Baeumer2009a} and \cite{Bala} for more details. Thus, for the additional case of $-A$ being a semigroup generator and $\alpha>1$, we just set $A^{\alpha}=f_{\alpha}(-A)$ as in \cite{Bala}.

The following theorem shows the rate of convergence for the Gr\"unwald formula approximating fractional powers of operators in this general setting. For the sake of notational simplicity we only give the first order version; the higher order version follows exactly along the same lines and is discussed in Corollary \ref{highOrder}.

\begin{theorem}\label{groupthm}
Let $X$ be a Banach space and $p\in \mathbb{N}$. Assume that $-A$ is the generator of a strongly continuous group, in case $p=0$, semigroup, of uniformly bounded linear operators $\{G(t)\}_{t\in \mathbb{R}}$ on $X$. 
Define
$$
\Phi^p_{\alpha,h} x=h^{-\alpha}\sum_{m=0}^{\infty} (-1)^{m}{\binom{\alpha}{m} }G((m-p)h)x=(-1)^{q+1}\psi_{\alpha,h,p}(-A)x, ~x\in X,
$$
where $\psi_{\alpha,h,p}(z)$ is given by \eqref{grunwaldmultgeneral} and $\alpha \in \mathbb{R}^+$ such that $2q-1 < \alpha < 2q+1, ~q\in \mathbb{N}.$

Then, as $h>0$, we have
\begin{equation}\label{grungroup}
\|\Phi^p_{\alpha,h}x-A^{\alpha}x\|\le C h \|A^{\alpha+1}x\|, ~x\in \mathcal{D}(A^{\alpha+1}).
\end{equation}
Furthermore, if $p=q,$ then $(-1)^{p+1}\Phi^p_{\alpha,h}$ generate $\{S^p_{\alpha,h}(t)\}_{t\ge 0}$, strongly continuous semigroups of linear operators on $X$ that are uniformly bounded in $h>0$ and $t\ge 0$ and uniformly  analytic in
$h>0$; i.e., 
there is $M>0$ such that $\|S^p_{\alpha,h}(t)\|_{\mathcal{B}(X)}\le M$ and $\|\Phi^p_{\alpha,h}S^p_{\alpha,h}(t)\|_{\mathcal{B}(X)}\le Mt^{-1}$ for all $t,h>0$.
\end{theorem}
\begin{proof}
If $\hat{g}$ is defined by $\hat{g}(z)=\frac{\omega_{p,\alpha}(-z)-1}{-z},~\Re z \le 0$, then  by Lemma \ref{ghatandghatprimeinl2} with $N=0$ and $\beta=1$ and \eqref{multipliernormequivalence} we have that
$
\|T_{k\mapsto h\hat{g}(ikh)}\|_{\mathcal{B}(L_1(\mathbb{R}))}\le Ch
$.
In case $p=0$, we have that $\supp(g)\subset \mathbb{R}_+$ and hence $\|T_{k\mapsto h\hat{g}(ikh)}\|_{\mathcal{B}(L_1(\mathbb{R_+}))}\le Ch$.
Therefore, by the transference estimate \eqref{trans},
$
\|h\hat{g}(-hA)\|_{\mathcal{B}(X)}\le Ch
$
for some $C>0$.
Thus, if $x\in \mathcal{D}(A^{\alpha+1})$, then
$
\|h\hat{g}(-hA)A^{\alpha+1}x\|\le Ch\|A^{\alpha+1}x\|
$.
Using the unbounded functional calculus developed in \cite{Baeumer2009a} (in case $p=0$ see \cite{Bala}), we have, for $x\in \mathcal{D}(A^{\alpha+1})$,
\begin{align*}
h\hat{g}(-hA)A^{\alpha+1}x&=\left[\left. h\hat{g}(hz) (-z)^{\alpha+1}\right|_{z=-A}\right]x\\
&=\left[\left. \left(h^{-\alpha}\sum_{m=0}^{\infty} (-1)^{m}{\binom{\alpha}{m} }e^{(m-p)h(z)}-(-z)^{\alpha}\right)\right|_{z=-A}\right]x\\
&=\Phi^p_{\alpha,h}x-A^{\alpha}x
\end{align*}
and the proof of \eqref{grungroup} is complete.

The strong continuity of $S^p_{\alpha,h}$ follow from Theorem \ref{grunwaldscsganalsg} and \cite[Theorem 4.1]{Baeumer2009a}, where the latter theorem establishes the transference of strong continuity, from $L_1(\mathbb{R})$ to a general Banach space $X$ (see, \cite[Theorem 5.1]{Bala} for the same result in the unilateral case). Finally, the operator norm estimates follow from the $L_1$-norm estimates in Theorem \ref{grunwaldscsganalsg} and the transference estimate \eqref{trans}, noting that by the functional calculus of \cite{Baeumer2009a} and \cite{Bala} it follows that $(\psi_{\alpha,h,p} e^{t\psi_{\alpha,h,p}})(-A)=(-1)^{p+1}\Phi^p_{\alpha,h}S^p_{\alpha,h}(t)$ for $t>0$ where $\psi_{\alpha,h,p}$ is given by \eqref{grunwaldmultgeneral} and $p=q$.
\end{proof}

The stability and consistency estimates of Theorem \ref{groupthm} allow us to obtain unconditionally convergent numerical schemes for the associated Cauchy problem in the abstract setting together with error estimates. To demonstrate this, we use the optimally shifted first order Gr\"unwald scheme as ``spatial'' approximation together with a first order scheme for time stepping, the Backward (Implicit) Euler scheme, to match the spatial order. Let $2p-1 < \alpha < 2p+1,~\alpha \in \mathbb{R}^+, ~p\in \mathbb{N},$ let $X$ be a Banach space and $-A$ be the generator of a uniformly bounded strongly continuous group (semigroup if $p=0$) of operators on $X$ and set $A_\alpha:=(-1)^{p+1}A^{\alpha}$. Consider the abstract Cauchy problem
$$
\dot{u}(t)=A_{\alpha}u(t); ~u(0)=x,
$$
with solution operator $\{S_{\alpha}(t)\}_{t\ge 0}$, where, as we already mentioned, $S_\alpha$ is a uniformly bounded analytic semigroup as shown in \cite[Theorem 4.6]{Baeumer2009a} and in \cite{Yos} for $0<\alpha<1$; that is when $-A$ is a semigroup generator. For its numerical approximation set
$$
\frac{u_{n+1}-u_{n}}{\tau}=(-1)^{p+1}\Phi^p_{\alpha,h}u_{n+1};~u_0=x, n=0,1,2,...;
$$
that is, with $A_{\alpha,h}:=(-1)^{p+1}\Phi^p_{\alpha,h}$,
\begin{equation*}\label{alg}
u_n=(I-\tau A_{\alpha,h})^{-n}x, n=1,2,...
\end{equation*}
We have the following smooth data error estimate.
\begin{theorem}\label{thm:conv}
Let $2p-1 < \alpha < 2p+1,~\alpha \in \mathbb{R}^+ , ~p\in \mathbb{N},~n \in \mathbb{N},$ and $0<\varepsilon\le 1.$ Let $X$ be a Banach space and $-A$ be the generator of a uniformly bounded strongly continuous group (semigroup if $p=0$) of operators on $X$ and set $t=n \tau$. If $x\in \mathcal{D}(A^{1+\varepsilon})$, then
\begin{equation}\label{err1}
\|S_\alpha(t)x-u_n\|\le C(n^{-1}\|x\|+h\frac{\alpha t^{\frac{\varepsilon}{\alpha}}}{\varepsilon}\|A^{1+\varepsilon}x\|),~n=1,2,...; t>0,
\end{equation}
and, if $x\in \mathcal{D}(A)$, then
\begin{equation}\label{err2}
\|S_\alpha(t)x-u_n\|\le C(n^{-1}\|x\|+(1+\alpha)h\left|\log\frac{t}{h^{\alpha}}\right|\,\|Ax\|),~n=1,2,...; t>0.
\end{equation}
\end{theorem}
\begin{proof}
To show \eqref{err1}, we split the error as
$$
S_\alpha(t)x-u_n=S_\alpha(t)x-S_{\alpha,h}(t)x+S^p_{\alpha,h}(t)x-u_n:=e_1+e_2.
$$
It was shown in Theorem \ref{groupthm} that $S^p_{\alpha,h}$ are bounded analytic semigroups on $X$, uniformly in $h$. Therefore,
$\|e_2\|\le Cn^{-1}\|x\|,~ n\in \mathbb{N},$
as shown in \cite{Stig}, with $C$ independent of $h$ and $t$. To bound $e_1$ we use the fact that all operators appearing commute being functions of $A$, to write
\begin{align}\label{e_1}
e_1=S_\alpha(t)x-S^p_{\alpha,h}(t)=\int_0^t(A_{\alpha}-A_{\alpha,h})S_{\alpha}(r)S^p_{\alpha,h}(t-r)x\,\dd r
\end{align}
Note that the analyticity of the semigroup $S_{\alpha}$ implies that there is a constant $M$ such that for $0\le \varepsilon\le 1$,
the estimate $\|A_{\alpha}^{1-\varepsilon}S_{\alpha}(t)\|\le Mt^{\varepsilon-1}$ holds for all $t>0$. Then, by Theorem \ref{groupthm},
\begin{align*}
\|e_1\|&\le Ch\int_{0}^t \|A^{\alpha+1}S_{\alpha}(r)S^p_{\alpha,h}(t-r)x\|\,\dd r\\&=Ch\int_0^t\|A_{\alpha}^{1-\frac{\varepsilon}{\alpha}}S_{\alpha}(r)
S^p_{\alpha,h}(t-r)A^{1+\varepsilon}x\|\,\dd r
\le Ch \frac{\alpha t^{\frac{\varepsilon}{\alpha}}}{\varepsilon}\|A^{1+\varepsilon}x\|,
\end{align*}
which completes the proof of \eqref{err1}.

To show \eqref{err2}, write $e_1$ in \eqref{e_1} as
$$
e_1=\int_0^{h^{\alpha}}(A_{\alpha}-A_{\alpha,h})S_{\alpha}(r)S^p_{\alpha,h}(t-r)x\,\dd r+\int_{h^\alpha}^t(A_{\alpha}-A_{\alpha,h})S_{\alpha}(r)S^p_{\alpha,h}(t-r)x\,\dd r
$$
It is already known that $A_{\alpha,h}x\to A_{\alpha}x$ as $h\to 0+$ for all $x\in \mathcal{D}(A_\alpha)$, see \cite[Proposition 4.9]{Baeumer2009a} and \cite{Westphal1974}, and hence we have stability
$\|A_{\alpha,h}x- A_{\alpha}x\|\le C\|A_{\alpha}x\|$ for all $x\in \mathcal{D}(A_\alpha)$. Therefore,
\begin{align*}
\|e_1\|\le & C\int_0^{h^{\alpha}}\|A_{\alpha}^{1-\frac{1}{\alpha}}S_{\alpha}(r)Ax\|\,\dd r+ Ch\left|\int_{h^\alpha}^t\|A_{\alpha}S_{\alpha}(r)Ax\|\,\dd r\right|\\
\le & C(h \alpha+h\left|\log\frac{t}{h^{\alpha}}\right|)\|Ax\|.
\end{align*}
\end{proof}
Note that the condition $x\in \mathcal{D}(A^{1+\varepsilon})$ in \eqref{err1} might be hard to check for $\epsilon\neq 1$, depending on $A$ and the Banach space $X$. However, one can always use $\varepsilon=1$ as $\mathcal{D}(A^2)$ is usually quite explicit.
We also obtain convergence and error estimates of stable higher order schemes (such as the second order Gr\"unwald formulae introduced in Section \ref{subsec:2order}).

\begin{corollary}\label{highOrder}
Let $\alpha \in \mathbb{R}_+$ with $2q-1<\alpha<2q+1$, $q\in \mathbb{N}$ and let
$$
\Psi_{\alpha,h}:=(-1)^{q+1}\sum_{j=0}^{N} b_j \Phi^{p_j}_{\alpha,c_jh}= \sum_{j=0}^N b_j \psi_{\alpha,c_j h, p_j}(-A)
$$
be an $N+1$-order Gr\"unwald approximation, where $\psi_{\alpha, h, p}(z)$ is given by \eqref{grunwaldmultgeneral} and  $b_j,c_j,p_j$ are as defined in \eqref{higherOrderA}. Assume the multiplier
$
\sum_{j=0}^{N} b_j\psi^{p_j}_{c_jh}(k),
$
where $\psi^p_h (k)$ is given by \eqref{grunwaldmultipliersymbol}, satisfies (i)-(iii) of Theorem \ref{maintheorem} with constants independent of $h$. If one solves the Cauchy problem
$$
\dot{u}(t)=\Psi_{\alpha,h}u(t); ~u(0)=x,
$$
with a strongly $A$-stable Runge-Kutta method of with stage order $s$ and order $r\ge s+1$, then, denoting the discrete solution by $u_n$ at time level $t=n\tau$,
$$
\|S_{\alpha}(t)x-u_n\|\le C\left(n^{-r}\|x\|+h^{N+1}\left|\log\frac{t}{h^{\alpha}}\right|\,\|A^{N+1}x\|\right),~h>0,~t=n \tau,
$$
for all $x\in \mathcal{D}(A^{N+1}).$

\end{corollary}

\begin{proof}
It is straight forward to see that Theorem \ref{groupthm} holds for $\Psi_{\alpha,h}$; i.e.,
 \begin{equation*}
\|\Psi_{\alpha,h}x-A^{\alpha}x\|\le C h^{N+1} \|A^{\alpha+N+1}x\|, ~x\in \mathcal{D}(A^{\alpha+N+1}),
\end{equation*}
$\|e^{t\Psi_{\alpha,h}}\|_{\mathcal B(X)}\le M$ and $\|\Psi_{\alpha,h}e^{t\Psi_{\alpha,h}}\|_{\mathcal B(X)}\le M/t$.
 Following the proof of Theorem \ref{thm:conv} we obtain the  ``spatial'' error estimate
\begin{equation}\label{eq:spee}
\|S_\alpha(t)x-e^{t\Psi_{\alpha,h}}x\|\le Ch^{N+1}\left|\log\frac{t}{h^{\alpha}}\right|\,\|A^{N+1}x\|.
\end{equation}
Since the analyticity of the semigroups $e^{t\Psi_{\alpha,h}}$ is uniform in $h$ (the constant $M$ does not depend on $h$), the statement follows from \cite[Theorem 3.2]{LO} (see also \cite{LeRoux}).
\end{proof}

\begin{remark}
The error estimates \eqref{err1} and \eqref{err2} are almost optimal in terms of the regularity of the data. We conjecture that one could remove the slight growth in $t$ from \eqref{err1} or the logarithmic factor in \eqref{err2} by considering the $L_1$-case again and using the theory of Fourier multipliers on Besov spaces. Then use the transference principle to derive the abstract result. We do not pursue this issue here any further.
\end{remark}

\begin{remark}\label{rem:interp}
The convergence rate given in Corollary \ref{highOrder} can be extended using the stability estimate
$$
\|S_\alpha(t)x-e^{t\Psi_{\alpha,h}}x\|\le C\|x\|
$$
and \eqref{eq:spee} to certain real interpolation spaces as in \cite[Corollary 4.4]{kov}. We note that while the spaces $\mathcal{D}(A^s)$ endowed with the graph norm are, in general, not interpolation spaces, they are embedded within appropriate interpolation spaces (see, for example, \cite[Corollary 6.6.3]{haase}) and therefore we obtain
$$
\|S_{\alpha}(t)x-u_n\|\le C\left(n^{-r}\|x\|+h^{s}\left|\log\frac{t}{h^{\alpha}}\right|^{\frac{s}{N+1}}\,\left(\|A^{s}x\|+\|x\|\right)\right),~h>0,~t=n \tau,
$$
for all $x\in \mathcal{D}(A^{s}),~s\in [0,N+1]$. Also note that we indeed have convergence of $u_n$ to $S_\alpha(t)x$ for all $x\in X$ as $\tau\to 0$ and $h\to 0$ by Lax's Equivalence Theorem as a consequence of stability and consistency. The order of convergence, however, might be very low depending on $x$.
\end{remark}

\section{Numerical Experiments}

In this section we give the results of two numerical experiments. The first is to explore the effect of the regularity of the initial distribution on the rate of convergence as needed for Corollary \ref{highOrder},  the other is to see how well a second and third order scheme fare in the numerical experiment done by Tadjeran et. al \cite{Tadjeran2006}.

\begin{figure}[htb]
  \includegraphics[width=12 cm]{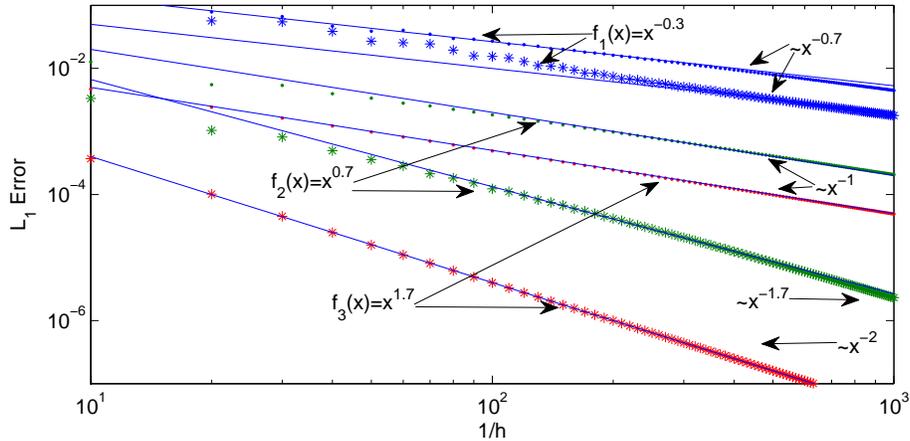}\\
  \caption{$L_1$-error for different initial conditions $f_i$ and a first ($\cdot$) and second order ($*$) scheme. Note the less than first order convergence for a ``bad'' initial condition; i.e. one that is not in the domain of $A$. Also note the less than second order convergence for a second order scheme but first order convergence for the first order scheme for an initial condition that is in the domain of $A$ but not in the domain of $A^2$.}\label{Fig:stable}
\end{figure}

\begin{example}
  We consider $X=L_1[0,1]$ and $A=(d/dx)$ with $$D(A)=\{f:f'\in L_1, f(0)=0\}$$ and $\alpha=0.8$. We approximate the solution to the Cauchy problems $$u'(t)=-A^\alpha u(t);u(0)=f_i, \;\; i=1,2,3$$ at $t=1$ with
  $$f_1(x)=x^{-0.3}, \;\;\;\;\;\;\; f_2(x)=x^{0.7},\;\;\;\;\;\;\; f_3(x)=x^{1.7},
   $$
with first and second order Gr\"unwald schemes (as in Proposition \ref{thm:mix}) as well as via a convolution of $f_i$ with an $\alpha$-stable density, which gives the exact solution but both the convolution and the density are computed numerically on a very fine grid. Note that $f_1\not\in D(A)$, $f_2\in D(A)$ but $f_2\not\in D(A^2)$ and $f_3\in D(A^2)$. However, $f_1\in D(A^{\beta_1})$ for $\beta_1<0.7$ and $f_2\in D(A^{\beta_2})$ for $\beta_2<1.7$. By Remark \ref{rem:interp} we expect about $0.7$-order convergence for both schemes in case of $u_0=f_1$, and first order convergence for the first order scheme for the other initial conditions. We expect about $1.7$-order convergence for the second order scheme in case of $u_0=f_2$ and second order convergence in case of $u_0=f_3$.
 For the temporal discretization we use MATLAB's ode45, a fourth-order Runge-Kutta method with a forced high degree of accuracy in order to investigate the pure spatial discretization error. We see in Figure \ref{Fig:stable} that we obtain the expected convergence in all cases.
\end{example}

\begin{example}
  Even though our theoretical framework is not directly applicable, because the fractional differential operator appearing in \eqref{tame} is defined on a finite domain with boundary conditions and has a multiplicative perturbation and hence it is not a fractional power of an auxiliary operator, we apply the second and third order approximations to the problem investigated by Tadjeran et al. \cite{Tadjeran2006}, namely approximating the solution to
  \begin{equation}\label{tame}
  \frac{\partial u(x,t)}{\partial t}=\frac{\Gamma(2.2)}{3!}x^{2.8}\frac{\partial^{1.8} u(x,t)}{\partial x^{1.8}}-(1+x)e^{-t}x^3; u(x,0)=x^3
  \end{equation}
  on the interval $[0,1]$ with boundary conditions $u(0,t)=0, u(1,t)=e^{-t}.$ The exact solution is given by $e^{-t}x^3$, which can be verified directly.

  A second order approximation of the fractional derivative is given by Proposition \ref{thm:mix}. In order to obtain a third order approximation we consider $$\phi_h(k)=a\psi^1_{h}+b\psi^{\frac 12}_{2h}+c\psi^0_{h}$$ with the coefficients $a,b$ and $c$ such that $\phi_h$ is a third order approximation; i.e.
  $$a=\frac{7-8\alpha+3\alpha^2}{3(\alpha-1)},\quad b=\frac{-7+3\alpha}{3(\alpha-1)},\quad c=1-a-b.$$ A quick plot of $\phi_h(k)$ for $k\in \R$ strengthens the conjecture that, for $\alpha=1.8$, the spectrum is in a sector in the left half plane and hence we expect stability and smoothing.

  We use again a fourth order Runge-Kutta method to solve the systems to $t=1$. Table 1 suggests that we indeed have second and third order convergence with respect to the spatial discretization parameter $\Delta x$.
 \begin{table}[htb]\label{Tab:Tad} \caption{Maximum error behaviour for second and third order Gr\"unwald approximations.}
  \begin{tabular}{|l|l|l|l|l|}
    \hline
    $\Delta x$ & Error $2^{nd}$ & Error rate & Error $3^{rd}$ & Error rate \\
    \hline
    $1/10$ & $6.825\times 10^{-5}$ & - & $9.180\times 10^{-6} $& - \\
    $1/15$ & $3.048\times 10^{-5}$ & $2.24\approx (15/10)^2 $& $1.933\times 10^{-6} $& $4.75>(15/10)^3$ \\
   $ 1/20$ & $1.708\times 10^{-5} $& $1.78\approx(20/15)^2 $& $7.825\times 10^{-7} $&$ 2.47\approx(20/15)^3 $\\
    $1/25$ & $1.088\times 10^{-5}$ & $1.57\approx(25/20)^2$ & $3.922\times 10^{-7}$ & $2\approx(25/20)^3$ \\
    \hline
  \end{tabular}
  \end{table}
\end{example}

\bibliographystyle{amsplain}

\providecommand{\bysame}{\leavevmode\hbox to3em{\hrulefill}\thinspace}
\providecommand{\MR}{\relax\ifhmode\unskip\space\fi MR }
\providecommand{\MRhref}[2]{%
  \href{http://www.ams.org/mathscinet-getitem?mr=#1}{#2}
}
\providecommand{\href}[2]{#2}

\end{document}